\pdfoutput=1
\RequirePackage{ifpdf}
\ifpdf 
\documentclass[pdftex]{sigma}
\else
\documentclass{sigma}
\fi

\usepackage{enumerate,mathrsfs}
\numberwithin{equation}{section}

\def\al{\alpha}
\def\de{\delta}
\def\Ga{\Gamma}
\def\R{\mathbb{R}}
\def\C{\mathbb{C}}
\def\N{\mathbb{N}}
\newcommand{\SU}{\mathrm{SU}}
\newcommand{\cA}{\mathscr{A}}

\newtheorem{Theorem}{Theorem}[section]
\newtheorem{Corollary}[Theorem]{Corollary}
\newtheorem{Lemma}[Theorem]{Lemma}

{ \theoremstyle{definition}

\newtheorem{Example}[Theorem]{Example}
\newtheorem{Remark}[Theorem]{Remark} }

\begin{document}

\allowdisplaybreaks

\renewcommand{\thefootnote}{$\star$}

\newcommand{\arXivNumber}{1509.06143}

\renewcommand{\PaperNumber}{008}

\FirstPageHeading

\ShortArticleName{Orthogonal vs.\ Non-Orthogonal Reducibility of Matrix-Valued Measures}

\ArticleName{Orthogonal vs.\ Non-Orthogonal Reducibility\\ of Matrix-Valued Measures\footnote{This paper is a~contribution to the Special Issue
on Orthogonal Polynomials, Special Functions and Applications.
The full collection is available at \href{http://www.emis.de/journals/SIGMA/OPSFA2015.html}{http://www.emis.de/journals/SIGMA/OPSFA2015.html}}}

\Author{Erik KOELINK~$^\dag$ and Pablo ROM\'AN~$^{\dag\ddag}$}

\AuthorNameForHeading{E.~Koelink and P.~Rom\'an}

\Address{$^\dag$~IMAPP, Radboud Universiteit, Heyendaalseweg 135, 6525 GL Nijmegen, The Netherlands}
\EmailD{\href{mailto:e.koelink@math.ru.nl}{e.koelink@math.ru.nl}}
\URLaddressD{\url{http://www.math.ru.nl/~koelink/}}

\Address{$^\ddag$~CIEM,
FaMAF, Universidad Nacional de C\'ordoba, Medina Allende s/n Ciudad Universitaria,\\
\hphantom{$^\ddag$}~C\'ordoba, Argentina}
\EmailD{\href{mailto:roman@famaf.unc.edu.ar}{roman@famaf.unc.edu.ar}}
\URLaddressD{\url{http://www.famaf.unc.edu.ar/~roman/}}

\ArticleDates{Received September 23, 2015, in f\/inal form January 21, 2016; Published online January 23, 2016}

\Abstract{A matrix-valued measure $\Theta$ reduces to measures of smaller size if there exists a constant invertible matrix $M$ such that $M\Theta M^*$ is  block diagonal. Equivalently, the real vector space $\cA$ of all matrices $T$ such that $T\Theta(X)=\Theta(X) T^*$ for any Borel set~$X$ is non-trivial. If the subspace $A_h$ of self-adjoints elements  in the commutant algebra $A$ of $\Theta$ is non-trivial, then $\Theta$ is reducible via a unitary matrix. In this paper we prove that $\cA$ is $*$-invariant if and only if $A_h=\cA$, i.e., every reduction of $\Theta$ can be performed via a unitary matrix.
The motivation for this paper comes from families of matrix-valued polynomials related to the group $\SU(2)\times \SU(2)$ and its quantum analogue. In both cases  the commutant algebra $A=A_h\oplus iA_h$ is of dimension two and the matrix-valued measures reduce unitarily into a~$2\times 2$ block diagonal matrix. Here we show that there is no further non-unitary reduction.}


\Classification{33D45; 42C05}

\renewcommand{\thefootnote}{\arabic{footnote}}
\setcounter{footnote}{0}

\section{Introduction}

The theory of matrix-valued orthogonal polynomials was initiated by Krein in 1949 and, since then, it was developed in dif\/ferent directions. From the perspective of the theory of orthogonal polynomials, one wants to study families of  truly matrix-valued orthogonal polynomials. Here is where the issue of reducibility comes into play. Given a matrix-valued measure, one can construct an equivalent measure by multiplying on the left by a constant invertible matrix and on the right by its adjoint. If the equivalent measure is a block diagonal matrix, then all the objects of interest (orthogonal polynomials, three-term recurrence relation, etc.) reduce to block diagonal matrices so that we could restrict to the study of the blocks of smaller size. An extreme situation occurs when the matrix-valued measure is equivalent to a diagonal matrix in which case we are, essentially, dealing with scalar orthogonal polynomials.

Our interest in the study of the reducibility of matrix-valued measures was triggered by
the families of matrix-valued orthogonal polynomials introduced in \cite{AKR15,KoeldlRR, KoelvPR1,KoelvPR2}.
In \cite{KoelvPR1} the study of the spherical functions of the group $\SU(2)\times \SU(2)$ leads to a matrix-valued measure~$\Theta$ and a sequence of matrix-valued orthogonal polynomials with respect to~$\Theta$. From group theoretical considerations, we were able to describe the symmetries of $\Theta$ and pinpoint two linearly independent matrices in the commutant of~$\Theta$, one being the identity. The proof that these matrices actually span the commutator required a careful computation. It then turns out that it is possible to conjugate~$\Theta$ with a constant unitary matrix to obtain a~$2\times 2$ block diagonal matrix.
An analogous situation holds true for a one-parameter extension of this example~\cite{KoeldlRR}. In~\cite{AKR15} from the study of the quantum analogue of $\SU(2)\times \SU(2)$ we constructed matrix-valued orthogonal polynomials which are matrix analogues of a subfamily of Askey--Wilson polynomials. The weight matrix can also be unitarily reduced to a $2\times 2$ block diagonal matrix in this case, again arising from quantum group theoretic considerations.

In \cite{Tira15}, the authors study non-unitary reducibility for matrix-valued measures and prove that a matrix-valued measure~$\Theta$ reduces into a block diagonal measure if the real vector space $\cA$ of all matrices $T$ such that $T\Theta(X)= \Theta(X) T^*$ for any Borel set~$X$ is not trivial, in contrast to the reducibility via unitary matrices that occurs when the commutant algebra of $\Theta$ is not trivial.

The aim of this paper is to develop a criterion to determine whether unitary and non-unitary reducibility of a weight matrix $W$ coincide in terms of the $*$-invariance of~$\cA$. Every reduction of~$\Theta$ can be performed via a unitary matrix if and only if $\cA$ is $*$-invariant, in which case $\cA=A_h$ where $A_h$ is the Hermitian part of the commutant of~$\Theta$, see Section~\ref{sec:reducibility_no_moments}. We apply our criterion to our examples \cite{AKR15,KoeldlRR, KoelvPR1} and we conclude that there is no further reduction than the one via a unitary matrix. We expect that a similar strategy can be applied to more general families of matrix-valued orthogonal polynomials as, for instance, the families related to compact Gelfand pairs given in \cite{HeckvP}. We also discuss an example where $\cA$ and $A_h$ are not equal.

It is worth noting that unitary reducibility strongly depends on the normalization of the matrix-valued measure. Indeed, if the matrix-valued measure is normalized by $\Theta(\R)=I$, then the real vector space $\cA$ is $*$-invariant and by our criterion, unitary and non-unitary reduction coincide. This is discussed in detail in Remark \ref{rmk:normalization}.

\section{Reducibility of matrix-valued measures}
\label{sec:reducibility_no_moments}
Let $M_N(\C)$ be the algebra of $N\times N$ complex matrices. Let $\mu$ be a $\sigma$-f\/inite positive measure on the real line and let the weight function $W\colon \mathbb{R} \to M_N(\C)$ be strictly positive def\/inite almost everywhere with respect to~$\mu$. Then
\begin{gather}
\label{eq:positive_measure_W}
\Theta(X)=\int_X W(x)d\mu(x),
\end{gather}
is a $M_N(\C)$-valued measure on $\R$, i.e., a function from the $\sigma$-algebra $\mathscr{B}$ of Borel subsets of $\R$ into the positive semi-def\/inite matrices in $M_N(\C)$ which is countably additive. Note  that any positive matrix measure can be obtained as in~\eqref{eq:positive_measure_W}, see for instance \cite[Theorem 1.12]{Berg}  and \cite{DamaPS}. More precisely, if $\tilde \Theta$ is a $M_N(\C)$-valued measure, and $\tilde \Theta_{\mathrm{tr}}$ denotes the scalar measure def\/ined by $\tilde \Theta_{\mathrm{tr}}(X)=\mathrm{Tr}(\tilde \Theta(X))$, then the matrix elements $\tilde \Theta_{ij}$ of $\tilde \Theta$ are absolutely continuous with respect to $\tilde \Theta_{\mathrm{tr}}$ so that,  by the Radon--Nikodym theorem, there exists a  positive def\/inite function $V$ such that
\begin{gather*}
d\tilde \Theta_{i,j}(x) = V(x)_{i,j} \, d\tilde \Theta_{\mathrm{tr}}(x).
\end{gather*}
Note that we do not require the normalization $\Theta(\R)=I$ as in \cite{DamaPS}. A detailed discussion about the role of  the normalization in the reducibility of the measure is given at the end of Section \ref{sec:reducibility_with_moments}.

Going back to the measure~\eqref{eq:positive_measure_W}, we have $d\Theta_{\mathrm{tr}}(x)=\mathrm{Tr}(W(x))\, d\mu(x)$ so that $\Theta_{\mathrm{tr}}$ is absolutely continuous with respect to~$\mu$.  Note that $\mathrm{Tr}(W(x))>0$ a.e.\ with respect to $\mu$ so that $\mu$ is absolutely continuous with respect to $\Theta_{\mathrm{tr}}$. The unicity of the Radon--Nikodym theorem implies $W(x)=V(x)\mathrm{Tr}(W(x))$, i.e.,~$W$ is a positive scalar multiple of $V$.

We say that two $M_N(\C)$-valued measures $\Theta_1$ and $\Theta_2$ are equivalent if there exists a constant nonsingular matrix $M$ such that $\Theta_1(X)=M\Theta_2(X)M^*$ for all $X\in \mathscr{B}$, where $*$ denotes the adjoint. A $M_N(\C)$-valued measure matrix $\Theta$ reduces to matrix-valued measures of smaller size if there exist positive matrix-valued measures $\Theta_1,\ldots, \Theta_m$ such that $\Theta$ is equivalent to the block diagonal matrix $\mathrm{diag}(\Theta_1(x),\Theta_2(x),\ldots,\Theta_m(x))$. If $\Theta$ is equivalent to a diagonal matrix, we say that $\Theta$ reduces to scalar measures. In \cite[Theorem~2.8]{Tira15}, the authors prove that a matrix-valued measure $\Theta$ reduces to matrix-valued measures of smaller size if and only if the real vector space
\begin{gather*}
\cA = \cA(\Theta) =\big\{ T\in M_{N}(\C)\,|\, T\Theta(X) = \Theta(X) T^* \,\,\, \forall\, X\in \mathscr{B} \big\},
\end{gather*}
contains, at least, one element which is not a multiple of the identity, i.e., $\R I \subsetneq \cA$, where $I$ is the identity. Note that our def\/inition of $\cA$ dif\/fers slightly from the one considered in~\cite{Tira15}. If~$W$ is a weight matrix for $\Theta$, then we require that $T\in \cA$ satisf\/ies $TW(x)=W(x)T^*$ almost everywhere with respect to~$\mu$.

If there exists a subspace $V\subset \C^N$ such that $\Theta(X)V \subset V$ for all $X\in\mathscr{B}$, since $\Theta(X)$ is self-adjoint for all $X\in\mathscr{B}$, it follows that $\Theta(X)V^\perp \subset V^\perp$ for all $X\in \mathscr{B}$. If $\iota_V\colon V\to \C^N$ is the embedding of $V$ into $\C^N$, then $P_V=\iota_V \iota_V^*\in M_N(\C)$ is the orthogonal projection on $V$ and satisf\/ies
\begin{gather*}
P_V\Theta(X) =\Theta(X)P_V, \qquad \text{for all }X\in\mathscr{B}.
\end{gather*}
Hence, the projections on invariant subspaces belong to the commutant algebra
\begin{gather*}
A=A(\Theta)=\big\{ T\in M_{N}(\C)\,|\,  T\Theta(X) = \Theta(X) T \,\,\, \forall\, X\in \mathscr{B}\big\}.
\end{gather*}
Since $\Theta(X)$ is self-adjoint for all $X\in\mathscr{B}$, $A$ is a unital $*$-algebra over $\C$.
We denote by $A_h$ the real subspace of $A$ consisting of all Hermitian matrices.  Then it follows that $A=A_h\oplus i A_h$. If $\C I \subsetneq A$, then there exists $T\in A_h$ such that $T\notin \C I$. The eigenspaces of $T$ for dif\/ferent eigenvalues are orthogonal invariant subspaces for $\Theta$. Therefore $\Theta$ is equivalent via a unitary matrix to matrix-valued measures of smaller size.

\begin{Remark}Let $S\in A$ and $T\in \mathscr{A}$. Then we observe that $S^*\in A$ and therefore $STS^*\Theta(x)=\Theta(x)ST^*S^*=\Theta(x)(STS^*)^*$ for all $X\in \mathscr{B}$. Hence there is an action from $A$ into $\mathscr{A}$ which is given by
\begin{gather*}S\cdot T=STS^*.\end{gather*}
\end{Remark}

\begin{Lemma}
\label{lem:no_skew_symmetric}
$\mathscr{A}$ does not contain non-zero skew-Hermitian elements.
\end{Lemma}
\begin{proof}
Suppose that $S\in \mathscr{A}$ is skew-Hermitian. Then $S$ is normal and thus unitarily diago\-na\-lizable, i.e., there exists a unitary matrix $U$ and a diagonal matrix $D=\mathrm{diag}(\lambda_1,\ldots,\lambda_N)$, $\lambda_i\in i\mathbb{R}$, such that $S=UDU^*$, see for instance \cite[Chapter~4]{Horn}. Since $S\in \cA$, we get
\begin{gather*}
D U^*\Theta(X)U = U^* \Theta(X) U D^* =  -U^* \Theta(X) U D ,\qquad \text{ for all } X\in \mathscr{B}.
\end{gather*}
The $(i,i)$-th entry of the previous equation is given by
\begin{gather*}
\lambda_i (U^*\Theta(X)U)_{i,i} = -(U^*\Theta(X)U)_{i,i} \lambda_i.
\end{gather*}
Take any $X_0\in \mathscr{B}$ such that $\Theta(X_0)$ is strictly positive def\/inite. Since $U$ is unitary, $U^*\Theta(X_0)U$ is strictly positive def\/inite and therefore $(U^*\Theta(X_0)U)_{i,i}>0$ for all $i=1,\ldots,N$, which implies that~$\lambda_i=0$.
\end{proof}

\begin{Theorem}
\label{thm:AcapA*}
$\cA \cap \cA^* = A_{h}$.
\end{Theorem}

\begin{proof}
Observe that if $T\in A_h$, then $T\Theta(X)=\Theta(X)T=\Theta(X)T^*$  for all $X\in \mathscr{B}$, and thus $A_h\subset \cA$. Since $T$ is self-adjoint, we also have $T=T^*\in \cA^*$.

On the other hand, let $T\in \cA \cap \cA^*$. Then $T^*\in \cA^* \cap \cA\subset \cA$, and since $\cA$ is a real vector space, $(T-T^*)\in \cA$. The matrix $(T-T^*)$ is skew-Hermitian and therefore by Lemma~\ref{lem:no_skew_symmetric} we have $(T-T^*)=0$. Hence~$T$ is self-adjoint and $T\in A_h$.
\end{proof}

\begin{Corollary}
\label{cor:AcapA_T=T*}
If $T\in \cA\cap\cA^*$, then $T=T^*$.
\end{Corollary}
\begin{proof}
The corollary follows directly from the proof of Theorem \ref{thm:AcapA*}.
\end{proof}

\begin{Corollary}
\label{cor:AHinvariant}
$\cA$ is $*$-invariant if and only if $\cA=A_h$.
\end{Corollary}
\begin{proof}
If $\cA=A_h$ then $\cA$ is trivially $*$-invariant. On the other hand, if we assume that
$\cA$ is $*$-invariant then the corollary follows directly from Theorem~\ref{thm:AcapA*}.
\end{proof}

\begin{Remark}
Corollary \ref{cor:AcapA_T=T*} says that if $\cA$ is $*$-invariant, then it is pointwise $*$-invariant, i.e., $T=T^*$ for all $T\in \cA$.
\end{Remark}

\begin{Remark}
Suppose that there exists $X\in \mathscr{B}$ such that $\Theta(X)\in\R_{>0} I$, then every $T\in \cA$ is self-adjoint and  Corollary \ref{cor:AHinvariant} holds true trivially.  Since $T\Theta(X)=\Theta(X)T^*$ for all $X\in \mathscr{B}$, if there is a point $x_0\in\mathrm{supp}(\mu)$ such that
\begin{gather*}
\lim_{\delta\to0} \left\| \frac{1}{\mu( (x_0-\delta, x_0+\delta) )} \Theta( (x_0-\delta,x_0+\delta)) - I \right\| = 0,
\end{gather*}
then it follows that $T=T^*$ and so Corollary~\ref{cor:AHinvariant} holds true. This is the case, for instance, for the examples given in \cite{alvarez2013orthogonal,AKdlR}, where $W(x_0)=I$ for some $x_0\in \mathrm{supp}(\mu)$. For Examples~\ref{ex:KdlRR} and~\ref{q-polynomials}, in general, there is no $x_0\in [-1,1]$ for which $W(x_0)=I$.
\end{Remark}

\section{Reducibility of matrix-valued orthogonal polynomials}
\label{sec:reducibility_with_moments}

Let $M_N(\C)[x]$ denote the set of $M_N(\C)$-valued polynomials in one variable~$x$. Let
$\mu$ be a f\/inite measure and $W$ be a weight matrix as in Section~\ref{sec:reducibility_no_moments}. In this section we assume that all the moments
\begin{gather*}
M_n=\int x^n W(x)\, d\mu(x), \qquad n\in\N,
\end{gather*}
exist and are f\/inite. Therefore we have a matrix-valued inner product on~$M_N(\C)$
\begin{gather*}
\langle P,Q \rangle = \int P(x)W(x)Q(x)^* \, d\mu(x), \qquad P,Q\in M_N(\C)[x],
\end{gather*}
where $*$ denotes the adjoint. By general considerations, e.g.,~\cite{DamaPS, GrunT}, it
follows that there exists a~unique sequence of monic
matrix-valued orthogonal polynomials $(P_n)_{n\in \N}$, where $P_n(x) = \sum\limits_{k=0}^n x^k P^n_{k}$ with $P^n_{k}\in  M_N(\C)$ and
$P^n_{n}=I$, the $N\times N$ identity matrix. The polynomials $P_n$ satisfy the orthogonality relations
\begin{gather*}
\langle P_n,P_m\rangle   =   \de_{nm} H_n, \qquad H_n\in M_N(\C),
\end{gather*}
where $H_n>0$ is the corresponding squared norm. Any other family $(Q_n)_{n\in\N}$ of matrix-valued orthogonal polynomials with respect to~$W$ is of the form $Q_n(x) = E_nP_n(x)$ for invertible matri\-ces~$E_n$. The monic orthogonal polynomials satisfy a three-term recurrence relation of the form
\begin{gather}
\label{eq:recrelat}
xP_n(x)=P_{n+1}(x)+B_{n}P_n(x)+C_nP_{n-1}(x), \qquad n\geq 0,
\end{gather}
where $P_{-1}=0$ and $B_n$, $C_n$ are matrices depending on $n$ and not on $x$.

\begin{Lemma}
\label{lem:symmetry_T}
Let $T\in \cA$. Then we have
\begin{enumerate}\itemsep=0pt
\item[$(1)$] The operator $T\colon M_{N}(\mathbb{C})[x]\to M_{N}(\mathbb{C})[x]$ given by $P\mapsto PT$ is symmetric with respect to~$\Theta$.
\item[$(2)$] $TP_n=P_nT$  for all $n\in \mathbb{N}$.
\item[$(3)$] $TH_n=H_nT^*$ for all $n\in \mathbb{N}$.
\item[$(4)$] $TM_n=M_nT^*$ for all $n\in \mathbb{N}$.
\item[$(5)$] $TB_n=B_nT$ and $TC_n=C_nT$ for all $n\in \N$.
\end{enumerate}
\end{Lemma}

\begin{proof}
Let $P,Q\in M_{N}(\mathbb{C})[x]$. Then
\begin{gather*}
\langle PT,Q \rangle  =\int P(x)T  W(x)   Q(x)^* \, d\mu(x)= \int P(x)   W(x)   T^*Q(x)^* \, d\mu(x)\\
\hphantom{\langle PT,Q \rangle}{} = \int P(x)  W(x)  (Q(x)T)^* \, d\mu(x) =\langle P,QT\rangle,
\end{gather*}
so that $T$ is a symmetric operator. This proves~(1). It follows directly from (1) that the monic matrix-valued orthogonal polynomials are eigenfunctions for the operator $T$, see, e.g.,  \cite[Proposition~2.10]{GrunT}. Thus, for every $n\in \mathbb{N}$ there exists a constant matrix $\Lambda_n(T)$ such that  $P_nT=\Lambda_n(T)P_n$. Equating the leading coef\/f\/icients of both sides of the last equation, and using that $P_n$ is monic, yields $T=\Lambda_n(T)$. This proves~(2).

The proof of (3) follows directly from (2) and the fact that $T\in \cA$. We have
\begin{gather*}
TH_n  = \int P_n(x)T  W(x)   P_n(x)^*\, d\mu(x) = \int P_n(x)  W(x)  T^*P_n(x)^* \, d\mu(x)\\
\hphantom{TH_n}{} =\int P_n(x)  W(x)   (P_n(x)T)^* d\mu(x)=\int P_n(x)  W(x)  (TP_n(x))^* \, d\mu(x)  =H_nT^*.
\end{gather*}
The proof of (4) is analogous to that of (3), replacing the polynomials~$P_n$ by~$x^n$. Finally, we multiply the three-term recurrence relation~\eqref{eq:recrelat} by $T$ on the left and on the right and we subtract both equations. Using that $TP_n=P_nT$ and $TP_{n+1}=P_{n+1}T$ we get
\begin{gather*}
TB_nP_n+TC_nP_{n-1}=B_nTP_n+C_nTP_{n-1}.
\end{gather*}
The coef\/f\/icient of $x^n$  is $TB_n=B_nT$ and therefore we also have $TC_n=C_nT$.
\end{proof}

Corollary \ref{cor:AHinvariant} provides a criterion to determine whether the set of Hermitian elements of the commutant algebra
$A$ is equal to $\cA$. However, for explicit examples, it might be cumbersome to verify the $*$-invariance of~$\cA$ from
the expression of the weight. Our strategy now is to view~$\cA$ as a subset of a, in general, larger set whose $*$-invariance can be established more easily and that implies the $*$-invariance of~$\cA$. Motivated by Lemma~\ref{lem:symmetry_T} we consider a sequence~$(\Gamma_n)_n$ of strictly positive def\/inite matrices such that if $T\in \cA$, then $T\Gamma_n=\Gamma_nT^*$ for all~$n$. Then for each $n\in \N$ and $\mathcal{I}\subset \N$ we introduce the $*$-algebras
\begin{gather*}
A^\Gamma_n=A(\Gamma_n)=\{ T\in M_{N}(\C)\,|\, T\Gamma_n = \Gamma_n T  \}, \qquad A^\Gamma_{\mathcal{I}}=\bigcap_{n\in \mathcal{I}} A^\Gamma_n,
\end{gather*}
and the real vector spaces
\begin{gather}
\label{eq:definition_Acal_n}
\cA^\Gamma_n= \cA(\Gamma_n) =\{ T\in M_{N}(\C)\,|\, T\Gamma_n = \Gamma_n T^*  \}, \qquad \cA^\Gamma_{\mathcal{I}}=\bigcap_{n\in \mathcal{I}} \cA^\Gamma_n.
\end{gather}
It is clear from the def\/inition that  $\cA \subset \cA^\Gamma_n$ for all $n\in \N$.
\begin{Remark}
\label{rmk:discrete_measure}
For any subset $\mathcal{I}\subset \N$, the sequence $(\Gamma_n)_n$ induces a discrete $M_N(\C)$-valued measure
supported on $\mathcal{I}$
\begin{gather*}
d\Gamma_\mathcal{I}(x)=\sum_{n\in \mathcal{I}} \Gamma_n  \delta_{n,x}.
\end{gather*}
Theorem \ref{thm:AcapA*} applied to the measure $d\Gamma_\mathcal{I}$ yields that $\cA^\Gamma_{\mathcal{I}}\cap (\cA^\Gamma_{\mathcal{I}})^*$ is the subset of Hermitian matrices in $A^\Gamma_{\mathcal{I}}$.
\end{Remark}

\begin{Theorem}
\label{thm:AHinvariant}
If $\cA^\Gamma_{\mathcal{I}}$ is $*$-invariant for some non-empty subset $\mathcal{I}\subset \N$, then $\cA=A_h$. In particular, the
statement holds true if $\cA^\Gamma_n$ is $*$-invariant for some $n\in \N$.
\end{Theorem}
\begin{proof}
If $T\in\cA$, then $T\in \cA^\Gamma_n$ for all $n\in \mathcal{I}$. Since $\cA^\Gamma_\mathcal{I}$ is $*$-invariant, then $T^* \in \cA^\Gamma_n$ for all $n\in \mathcal{I}$. If we apply Corollary \ref{cor:AcapA_T=T*} to the measure in Remark \ref{rmk:discrete_measure}, we obtain $T=T^*$. Therefore $T\in A_h\subset \cA$ and thus $\cA$ is $*$-invariant. Hence the theorem follows from Corollary \ref{cor:AHinvariant}.
\end{proof}

\begin{Remark}
Two obvious candidates for sequences $(\Gamma_n)_n$ are given in Lemma \ref{lem:symmetry_T}, namely the sequence of squared norms $(H_n)_n$ and the sequence of even moments $(M_{2n})_{n}$.
\end{Remark}

\begin{Remark}
\label{rmk:SWS}
Let $\Theta$ be a $M_N(\C)$-valued measure, not necessarily with f\/inite moments and take a~positive def\/inite matrix $\Gamma$ such that $T\Gamma=\Gamma T^*$ for all $T\in \cA(\Theta)$. Let $S$ be a~positive def\/inite matrix such that $\Gamma=S^2$. We can now consider the $M_N(\C)$-valued measure $S^{-1}\Theta S^{-1}$. By a~simple computation, we check that  $T\in \cA(\Theta)$ if and only if $S^{-1}T S\in \cA(S^{-1}\Theta S^{-1})$. This gives
\begin{gather*}
\cA\big(S^{-1}\Theta S^{-1}\big)=S^{-1}  \cA(\Theta)   S.
\end{gather*}
Moreover, if $T\in \cA(\Theta)$, then $T\Gamma=\Gamma T^*$ implies that
$S^{-1}TS=ST^*S^{-1}=(S^{-1}TS)^*$. Hence $S^{-1}TS$ is self-adjoint for all $T\in \cA(\Theta)$. Then we have by Corollary \ref{cor:AHinvariant}
\begin{gather*}
A\big(S^{-1}\Theta S^{-1}\big)_h = \cA\big(S^{-1}\Theta S^{-1}\big).
\end{gather*}
On the other hand, if $U\in A_h(\Theta)$, then
\begin{gather*}
S^{-1}US S^{-1}\Theta(X)S^{-1} =S^{-1}U\Theta(X)S^{-1}=S^{-1}\Theta(X)S^{-1}SUS^{-1}\\
\hphantom{S^{-1}US S^{-1}\Theta(X)S^{-1}}{} =S^{-1}\Theta(X)S^{-1}\big(S^{-1}US\big)^*,
\end{gather*}
for all $X\in \mathscr{B}$. Therefore $S^{-1} A(\Theta)_h S \subset \mathscr{A}(S^{-1}\Theta S^{-1})=A(S^{-1} \Theta S^{-1})_h$. In general this is an inclusion, see Example~\ref{ex:one}.
\end{Remark}

\begin{Remark}
\label{rmk:normalization1}
Suppose that $\Theta$ is a $M_N(\C)$-valued measure with f\/inite moments and that~$M_{2n} \in \R_{>0} I$, respectively $H_n\in \R_{>0} I$, for some $n\in \N$. Then it follows from Lemma~\ref{lem:symmetry_T} and~\eqref{eq:definition_Acal_n}  that $T=T^*$ for all $T\in\cA^M_{2n}$, respectively $T\in \cA^N_n$. Then Theorem~\ref{thm:AHinvariant} says that $\cA=A_h$.
\end{Remark}

\begin{Remark}
\label{rmk:normalization}
Let $\Theta$ be a $M_N(\C)$-valued measure such that the f\/irst moment $M_0$ is f\/inite. Then there exists a positive def\/inite matrix $S$ such that
$M_0=S^2$. The measure $\widetilde \Theta = S^{-1}\Theta S^{-1}$ satisf\/ies
\begin{gather*}
\widetilde \Theta(\R) = S^{-1}\Theta(\R)S^{-1}= S^{-1}   M_0   S^{-1} = I.
\end{gather*}
Therefore by Remark \ref{rmk:normalization1} we have that $\cA(S^{-1}\Theta S^{-1}) = A(S^{-1}\Theta S^{-1})_h$. Observe that the normalization $\widetilde \Theta(\R)=I$ is assumed in \cite{DamaPS} so that in the setting of that paper the real subspace of Hermitian elements in the commutant coincides with the real vector space $\cA(\Theta)$.
\end{Remark}

\section{Examples}

In this section we discuss three examples of matrix-valued weights that exhibit dif\/ferent features. The f\/irst example is a slight variation of \cite[Example~2.6]{Tira15}.
\begin{Example}
\label{ex:one}
Let $\mu$ be the Lebesgue measure on the interval $[0,1]$, and let $W$ be the weight
\begin{gather*}
W(x)=\begin{pmatrix} x^2+x & x \\ x & x \end{pmatrix} = \begin{pmatrix} 1 & \frac{\sqrt{6}}{3} \\ 0 & 1 \end{pmatrix} \begin{pmatrix} x^2 & 0 \\ 0 & x \end{pmatrix}
\begin{pmatrix} 1 & 0 \\ \frac{\sqrt{6}}{3} & 1 \end{pmatrix}.
\end{gather*}
A simple computation gives that $A$ and $\cA$ are given by
\begin{gather*}
A=\C I, \qquad \cA=\R \begin{pmatrix} 1 & 0 \\ 0 & 1 \end{pmatrix} + \R \begin{pmatrix} 1 & -\frac{\sqrt{6}}{3} \\ 0 & 0 \end{pmatrix}.
\end{gather*}
Observe that $\cA$ is clearly not $*$-invariant since $\left(\begin{smallmatrix} 1 & -\frac{\sqrt{6}}{3} \\ 0 & 0 \end{smallmatrix}\right)^*\notin \cA$. Now we consider the sequen\-ce~$(M_{2n})_n$ of even moments. The f\/irst moment is given by $M_0=\left(\begin{smallmatrix} \frac23 & \frac{\sqrt{6}}{6} \\  \frac{\sqrt{6}}{6} & \frac12 \end{smallmatrix}\right)$ and the algebras~$A^M_0$ and~$\cA^M_0$ are
\begin{gather}
A^M_0 = \C  \begin{pmatrix} 1 & 0 \\ 0 & 1 \end{pmatrix}  + \C  \begin{pmatrix} \frac{\sqrt{6}}{6} & 1 \\ 1 & 0 \end{pmatrix},\nonumber \\
\cA^M_0 = \R \begin{pmatrix} 1 & 0 \\ 0 & 1 \end{pmatrix}  + \R  \begin{pmatrix} \frac{\sqrt{6}}{6} & 1 \\ 1 & 0 \end{pmatrix} + \R \begin{pmatrix} 1 & -\frac{\sqrt{3}}{6} \\ 0 & 0 \end{pmatrix} + i \R \begin{pmatrix} 1 & -\frac{2\sqrt{6}}{3} \\ \frac{\sqrt{6}}{2} & -1 \end{pmatrix}.\label{eq:example_ignacio_AM}
\end{gather}
This gives the inclusions $\R I \subsetneq (A^M_0)_h \subsetneq \cA^M_0$. It  is also clear from \eqref{eq:example_ignacio_AM} that
$\cA^M_0 \cap (\cA^M_0)^* = (A^M_0)_h$.

Now we proceed as in Remark \ref{rmk:normalization}, we take the positive def\/inite matrix $S$ such that $M_0=S^2$. Here $S=\frac{1}{15}\left(\begin{smallmatrix}\sqrt{6}+9 & 3\sqrt{6}-3 \\  3\sqrt{6}-3 & \frac32\sqrt{6}+6   \end{smallmatrix}\right)$. Then
\begin{gather*}
S^{-1}W(x)S^{-1}=
\frac{1}{25} \begin{pmatrix} (33+12\sqrt{6})x^2+(28-8\sqrt{6})x &  -(6+9\sqrt{6})x^2+(4+6\sqrt{6})x \\ -(6+9\sqrt{6})x^2+(4+6\sqrt{6})x &  (42-12\sqrt{6})x^2+(22+8\sqrt{6})x\end{pmatrix}.
\end{gather*}
We f\/inally have that
\begin{gather*}
\cA(S^{-1}\Theta S^{-1})= \R  I + \R  E +\R  F, \quad \cA(S^{-1}M_0S^{-1}) = \R I + \R  E+ \R  F + \R  G,
\end{gather*}
where
\begin{gather*}E= S^{-1} \begin{pmatrix} \frac{\sqrt{6}}{6} & 1 \\ 1 & 0 \end{pmatrix} S = \begin{pmatrix} \frac{\sqrt{6}}{6} & 1 \\ 1 & 0 \end{pmatrix}, \qquad G =iS^{-1} \begin{pmatrix} 1 & -\frac{2\sqrt{6}}{3} \\ \frac{\sqrt{6}}{2} & -1 \end{pmatrix} S = \begin{pmatrix} 0 & -i \\ i & 0 \end{pmatrix},\\
F= S^{-1}  \begin{pmatrix} 1 & -\frac{\sqrt{3}}{6} \\ 0 & 0 \end{pmatrix}   S = \frac{1}{25} \begin{pmatrix} 11+4\sqrt{6} & -(2+3\sqrt{6}) \\ -(2+3\sqrt{6}) & 14-4\sqrt{6} \end{pmatrix}.
\end{gather*}
Then we have the following inclusions:
\begin{gather*}
\R I = S^{-1}A(\Theta)_h S \subsetneq A\big(S^{-1}\Theta S^{-1}\big)_h= \cA(S^{-1}\Theta S^{-1}),
\end{gather*}
and
\begin{gather*}
S^{-1}\big(A^M_0\big)_h S \subsetneq A\big(S^{-1}M_0S^{-1}\big)_h = \cA \big(S^{-1}M_0S^{-1}\big).
\end{gather*}
\end{Example}

\begin{Example}
\label{ex:KdlRR}

Our second example is a family of matrix-valued Gegenbauer polynomials introduced in \cite{KoeldlRR}. For $\ell\in\frac12\N$ and $\nu> 0$, let
$d\mu(x)=(1-x^2)^{\nu-1/2}dx$ where $dx$ is the Lebesgue measure on $[-1,1]$  and let $W^{(\nu)}$ be the $(2\ell+1)\times(2\ell+1)$ matrix
\begin{gather*}
\big(W^{(\nu)}(x)\big)_{m,n}=  \sum_{t=\max(0, n+m-2\ell)}^{m}
 \al_{t}^{(\nu)}(m,n)
C_{m+n-2t}^{(\nu)}(x),\\
\alpha_{t}^{(\nu)}(m,n) = (-1)^{m}
\frac{n! m! (m+n-2 t)!}{t! (2\nu)_{m+n-2t} (\nu)_{n+m-t}}
\frac{(\nu)_{n-t} (\nu)_{m-t}}{(n-t)! (m-t)!}\frac{(n+m-2t+\nu)}{(n+m-t+\nu)}\\
\hphantom{\alpha_{t}^{(\nu)}(m,n) =}{}
 \times (2\ell-m)!(n-2\ell)_{m-t}(-2\ell-\nu)_{t} \frac{(2\ell+\nu)}{(2\ell)!} ,
\end{gather*}
where $n,m\in\{0,1,\dots, 2\ell\}$ and $n\geq m$. The matrix $W^{(\nu)}$ is extended to a~symmetric matrix,
$\left(W^{(\nu)}(x)\right)_{m,n}=\left(W^{(\nu)}(x)\right)_{n,m}$. In \cite[Proposition~2.6]{KoeldlRR} we proved
that~$A$ is generated by the identity matrix $I$ and the involution $J\in M_{2\ell+1}(\C)$ def\/ined by $e_j \mapsto e_{2\ell-j}$

Now we will use  Theorem~\ref{thm:AHinvariant} to prove that~$A_h=\mathscr{A}$. This says that there is no further non-unitary
reduction of the weight $W$. As a sequence of positive def\/inite matrices we take the squared norms of the monic polynomials,
$(\Gamma_n)_n=(H_n)_n$, that were explicitly calculated in \cite[Theorem~3.7]{KoeldlRR} and are given by the following diagonal matrices
\begin{gather*}
\bigl( H^{(\nu)}_n\bigr)_{i,k} = \delta_{i,k}  \sqrt{\pi}  \frac{\Ga(\nu+\frac12)}{\Ga(\nu+1)}
\frac{\nu(2\ell+\nu+n)}{\nu+n} \frac{n! (\ell+\frac12+\nu)_n (2\ell+\nu)_n}{(\nu+k)_n(2\ell+2\nu+n)_n(2\ell+\nu-k)_n} \\
\hphantom{\bigl( H^{(\nu)}_n\bigr)_{i,k} =}{}
\times \frac{k! (\ell+\nu)_n(2\ell-k)! (n+\nu+1)_{2\ell}}{(2\ell+\nu+1)_n(2\ell)! (n+\nu+1)_k (n+\nu+1)_{2\ell-k}}.
\end{gather*}
For any $n\in \N$ we choose $\mathcal{I}=\{n,n+1\}$. If we take $T\in  \cA^\Gamma_{\mathcal{I}}$, i.e., $TH^{(\nu)}_n=H^{(\nu)}_nT^*$ and $TH^{(\nu)}_{n+1}=H^{(\nu)}_{n+1}T^*$, it follows that
\begin{gather*}
T_{i,j} =\frac{\big(H^{(\nu)}_n\big)_{j,j}}{(H^{(\nu)}_n)_{i,i}} \overline T_{j,i}
   = \frac{j!(2\ell-j)!(\nu+i)_n(2\ell+\nu-i)_n(n+\nu+1)_i(n+\nu+1)_{2\ell-i}}{i!(2\ell-i)!(\nu+j)_n(2\ell+\nu-j)_n(n+\nu+1)_j(n+\nu+1)_{2\ell-j}} \overline T_{j,i}.
\end{gather*}
It follows directly from this equation that $T_{i,i}\in \mathbb{R}$ and $T_{i,2\ell-i} =\overline T_{2\ell-i,i}$. Now we observe that
\begin{gather}
 T_{i,j}= \frac{\big(H^{(\nu)}_n\big)_{j,j}}{\big(H^{(\nu)}_n\big)_{i,i}} \overline T_{j,i} = \frac{\big(H^{(\nu)}_n\big)_{j,j}}{(H^{(\nu)}_n)_{i,i}} \frac{\big(H^{(\nu)}_{n+1}\big)_{i,i}}{\big(H^{(\nu)}_{n+1}\big)_{j,j}}    T_{i,j} \nonumber\\
\hphantom{T_{i,j}}{}  =\frac{(\nu+j+n)(\nu+j+n+1)(2\ell+n+\nu-j)(2\ell+n+\nu+1-j)}{(\nu+i+n)(\nu+i+n+1)(2\ell+n+\nu-i)(2\ell+n+\nu+1-i)}  T_{i,j}.\label{eq:condition_T}
\end{gather}
Equation~\eqref{eq:condition_T} implies that $T_{i,j}=0$ unless $j=i$ or $j=2\ell-i$. Hence~$T$ is self-adjoint and thus~$\cA^\Gamma_\mathcal{I}$ is $*$-invariant and from Theorem~\ref{thm:AHinvariant} we have~$\cA=A_h$. We conclude that~$\cA$ is the real span of~$\{I,J\}$, and so there is no further non-unitary reduction.
\end{Example}

\begin{Example}
\label{q-polynomials}
Our last example is a $q$-analogue of the previous example for $\nu=1$. This sequence of  matrix-valued orthogonal polynomials
 matrix analogues of a subfamily of Askey--Wilson polynomials and were obtained by studying matrix-valued spherical functions
 related to the quantum analogue of $\SU(2)\times \SU(2)$.  For any $\ell\in \frac12 \N$ and $0<q<1$, we have the measure $d\mu(x)=\sqrt{1-x^2}\, dx$ supported on $[-1,1]$ and a $(2\ell+1)\times (2\ell+1)$ weight matrix $W$ which is given in \cite[Theorem~4.8]{AKR15}, we omit the explicit expression here and we give, instead, the explicit expression for the squared norms $H_n$ of the monic orthogonal polynomials
\begin{gather*}
(H_n)_{i, j}  = \delta_{i, j}
\frac{q^{-2\ell}2^{-2n} (q^2,q^{4\ell+4};q^2)^2_n (1-q^{4\ell+2})^2}{(q^{2i+2},q^{4\ell-2i+2};q^2)^2_n(1-q^{2n+2i+2})(1-q^{4\ell-2i+2n+2})}.
\end{gather*}
The expression for $H_n$ is obtained combining Theorem~4.8 and Corollary~4.9 of~\cite{AKR15}. The commutant algebra is
generated by $\{I,J\}$ as in the previous example, see \cite[Proposition~4.10]{AKR15}. We take $(\Gamma_n)_n=(H_n)_n$, $\mathcal{I}=\{n,n+1\}$ for any $n\in \N$ and observe that $TH_n=H_nT^*$ and $TH_{n+1}=H_{n+1}T^*$ implies
\begin{gather*}
T_{i,j}  = \frac{\big(H^{(\nu)}_n\big)_{j,j}}{\big(H^{(\nu)}_n\big)_{i,i}}  \overline T_{j,i}= \frac{\big(H^{(\nu)}_n\big)_{j,j}}{\big(H^{(\nu)}_n\big)_{i,i}} \frac{\big(H^{(\nu)}_{n+1}\big)_{i,i}}{\big(H^{(\nu)}_{n+1}\big)_{j,j}}   T_{i,j}\\
\hphantom{T_{i,j}}{}
=\frac{(1-q^{2n+2j+2})(1-q^{2n+2j+4})(1-q^{4\ell+2n-2i+2})(1-q^{4\ell+2n-2i+4})}{(1-q^{2n+2i+2})(1-q^{2n+2i+4})(1-q^{4\ell+2n-2j+2})(1-q^{4\ell+2n-2j+4})}T_{i,j}.
\end{gather*}
As in the previous example, it follows that~$T$ is self-adjoint and therefore, $\cA=A_h$. Hence there is no further non-unitary reduction  for~$W$.
\end{Example}

\begin{Remark}
Theorem \ref{thm:AHinvariant} can be used to determine the irreducibility of a weight matrix. In fact, with the commutant algebras already determined in~\cite{KoeldlRR} and~\cite{AKR15},
Theorem~\ref{thm:AHinvariant}  implies that the restrictions of the weight matrices of Examples~\ref{ex:KdlRR} and~\ref{q-polynomials} to the eigenspaces of the matrix~$J$ are irreducible.  For some explicit cases see \cite[Section~8]{KoelvPR1}.
\end{Remark}

\subsection*{Acknowledgements}

We thank I.~Zurri\'an for pointing out a similar example to Example~\ref{ex:one} to the f\/irst author.
The research of Pablo Rom\'an is supported by the Radboud Excellence Fellowship.
We would like to thank the anonymous referees for their comments and remarks, that have helped us to improve the paper.

\pdfbookmark[1]{References}{ref}
\LastPageEnding

\end{document}